\documentclass[reqno,11pt]{amsart}
\usepackage{graphicx} % Required for inserting images
\usepackage{tikz-cd}
\usepackage{hyperref}
\usepackage{soul}
\usepackage{mathrsfs}  

\usepackage{enumitem}
\usepackage{marginnote}
\usepackage[nobysame]{amsrefs}
\usepackage{hyperref}
\usepackage[scr=rsfso]{mathalfa}
\usepackage{xcolor}
\usepackage{slashed}
\usepackage{graphicx}
\usepackage{subcaption}
\usepackage{mathtools}

\usepackage[nobysame]{amsrefs}
\usepackage{comment}
\usepackage{soul}
\usepackage{tikz}
\usepackage{marginnote}
\usepackage{anysize}
\usepackage{bbm}
\usepackage{cancel}
\usepackage[normalem]{ulem}
\usepackage[percent]{overpic} 

\numberwithin{equation}{section}

\theoremstyle{plain}
\newtheorem{theorem}[subsection]{Theorem}
\newtheorem{lemma}[subsection]{Lemma}
\newtheorem{corollary}[subsection]{Corollary}

\theoremstyle{definition}
\newtheorem{definition}[subsection]{Definition}

\theoremstyle{remark}
\newtheorem{remark}[subsection]{Remark}

\newcommand{\ind}{\operatorname{ind}}
\newcommand{\id}{\operatorname{id}}
\newcommand{\dom}{\operatorname{dom}}

\newcommand{\dm}{{\partial M}}
\newcommand{\CC}{\mathbb{C}}
\newcommand{\RR}{\mathbb{R}}
\newcommand{\cRR}{\mathscr{R}}
\newcommand{\ZZ}{\mathbb{Z}}

\newcommand{\cHH}{\check{H}}
\newcommand{\hHH}{\hat{H}}

\newcommand{\bfu}{\mathbf{u}}
\newcommand{\bfv}{\mathbf{v}}
\newcommand{\upper}{\uppercase\expandafter}

\newcommand{\n}{\nabla}
\newcommand{\p}{\partial}
\newcommand{\pM}{{\p M}}

\newcommand{\End}{\operatorname{End}}
\newcommand{\ad}{{\rm ad}}

\newcommand{\oB}{\bar{B}}

\newcommand{\coker}{\operatorname{coker}}
\newcommand{\ch}{\operatorname{ch}}

\renewcommand{\>}{\rangle}
\newcommand{\<}{\langle}

\renewcommand{\det}{\operatorname{det}}

\newcommand{\calB}{\mathscr{B}}

\newcommand{\calE}{\mathscr{E}}

\newcommand{\D}{\mathscr{D}}

\newcommand{\E}{\mathscr{E}}

\newcommand{\ID}{\operatorname{Id}}

\newcommand{\tr}{{\operatorname{tr}}}

\newcommand{\RP}{\RR\mathrm{P}}

\newcommand{\hM}{\hat{M}}
\newcommand{\hE}{\hat{E}}
\newcommand{\hD}{\hat{D}}
\newcommand{\hA}{\hat{A}}

\begin{document}
\title{A gluing formula for the $\ZZ_2$-valued index of odd symmetric operators}
\author{Maxim Braverman} 
%thanks{Supported in part by the NSF grant DMS-1005888.}
\address{Department of Mathematics,
        Northeastern University,
        Boston, MA 02115,
        USA
         }
\email{m.braverman@northeastern.edu}
\author{Ahmad Reza Haj Saeedi Sadegh }
    \address{Department of Mathematics, Dartmouth College, Hanover, NH 03755,
USA}
\email{arhsaeedi@gmail.com}

\author{Junrong Yan}
    \address{Department of Mathematics,
        Northeastern University,
        Boston, MA 02115,
        USA}
\email{j.yan@northeastern.edu}
%\date{\today}
\begin{abstract}
We investigate Dirac-type operator $D$ on involutive manifolds with boundary with symmetry, which forces the index of $D$ to vanish. We study the secondary $\ZZ_2$-valued index of elliptic boundary value problems for such operators. We prove a $\ZZ_2$-valued analog of the splitting theorem: the $\ZZ_2$-valued index of an operator on a closed manifold $M$ equals the $\ZZ_2$-valued index of a boundary value problem on a manifold obtained by cutting $M$ along a hypersurface $N$. When $N$ divides $M$ into two disjoint submanifolds $M_1$ and $M_2$, the $\ZZ_2$-valued index on $M$ is equal to the mod 2 reduction of the usual $\ZZ$-valued index of the Atiyah-Patodi-Singer boundary value problem on $M_1$. This leads to a cohomological formula for the $\ZZ_2$-valued index. 
\end{abstract}


\maketitle

%------------------------------------------------------------
%------------------------------------------------------------
\section{Introduction}

Atiyah and Singer, \cite{AtSinger69}, initiated the study of the index theory of real operators, which commute with an action of the Clifford algebra $Cl_k$.  The special case of operators commuting with $Cl_2$ is equivalent to the study of so-called {\em odd symmetric} operators. These are the operators on a complex Hilbert space $H$ with an anti-unitary anti-involution, i.e., an anti-linear map  $\tau$ with $\tau^{-1}= \tau^*= -\tau$. A Fredholm operator $D$ is called \textit{odd symmetric} if $\tau D \tau^{-1} = D^*$. This case attracted particular attention because of its connection to physical systems with time-reversal symmetry \cites{Schulz-Baldes15,DeNittisSB15,DollSB21,GrafPorta13,Hayashi17}. The usual index of an odd symmetric operator vanishes, but Atiyah and Singer defined a secondary $\ZZ_2$-valued index, called the $\tau$-index,  as the dimension of the kernel of $D$ modulo 2:
\begin{equation}\label{E:Itauindex}
    \ind_\tau D\ := \ \dim\ker D \qquad \text{mod}\quad 2.
\end{equation}

In \cites{BrSaeedi24deformation,BrSaeedi24spflow}, the first and the second authors extended several classical index-theoretical results to the $\tau$-index of odd symmetric differential operators on a manifold $M$ with an involution $\theta^M:M\to M$. They also gave an extension and a new interpretation of the $\ZZ_2$-valued bulk index for topological insulators of type AII, \cite{GrafPorta13}, and a generalization of bulk-edge correspondence for $\ZZ_2$-valued invariants (this should be compared with \cite{Br19Toeplitz}, where similar results were obtained for usual $\ZZ$-valued invariants).

In the present paper, we discuss boundary value problems for odd symmetric operators. We obtain a $\ZZ_2$-valued analogue of the classical splitting theorem, \cite[Theorem 8.17]{BarBallmann12}, which relates the index of $D$ with the index of the Atiyah-Patodi-Singer (APS) boundary value problem on a manifold obtained by cutting $M$ along a hypersurface. If a manifold $M$ with involution $\theta^M$ is divided into two disjoint manifolds $M_1$ and $M_2$ by a $\theta^M$-invariant hypersurface, then we show that the $\ZZ_2$-index of an odd symmetric Dirac-type operator $D$ on $M$ is equal modulo 2 to the usual $\ZZ$-valued index of the APS boundary value problem of the restriction of $D$ to $M_1$. This allows us to express $\ind_\tau D$ as an integral of a differential form over $M_1$.

Let us give more details about our results.

An {\em involutive manifold} $(M,\theta^M)$ is a manifold $M$ together with a smooth involution $\theta^M:M\to M$. A \textit{quaternionic bundle} over $M$  is a complex Hermitian vector bundle $E$ together with an anti-unitary anti-involution $\theta^E:E\to E$, which covers $\theta^M$,  \cites{Dupont69,DeNittisGomi15,Hayashi17}. We define an anti-unitary anti-involution $\tau:\Gamma(M,E)\to \Gamma(M,E)$ by 
\[
\tau:\, f(x)\ \mapsto \ \theta^E f(\theta^M x), \qquad f\in\Gamma(M,E),
\]
and consider a Dirac-type operator $D$ which is odd symmetric with respect to $\tau$. If $M$ is compact, this operator is Fredholm. If $M$ is a compact manifold with boundary we consider an elliptic boundary condition $B\subset H^{-1/2}(\p M,E_{\p M})$ for $D$ and denote by $D_B$ the restriction of $D$ to the sections of $E$, whose restrictions to $\p M$ lie in $B$. Then $D_B$ is Fredholm. If $\tau$ interchange $B$ with the {\em adjoint boundary conditions} $B^\ad$, then $D_B$ is odd symmetric, and we are interested in its $\tau$-index \eqref{E:Itauindex}.

Suppose $N\subset M$ is a $\theta^M$-invariant hypersurface such that a small neighborhood $U$ of $N$ decomposes as $U= U_1\sqcup_N U_2$ with $\theta^M(U_1)= U_2$. We cut $M$ along $N$ and obtain a new manifold $\hat{M}$, whose boundary consists of two copies of $N$, denoted by $N_1$ and $N_2$. 

Let $D$ be an odd symmetric Dirac-type operator acting on sections of a quaternionic Dirac bundle $E$ over $M$. 
Let $\hat{D}$ denote the induced Dirac-type operator on $\hat{M}$. Let $\hat{B}$ be the boundary condition for $\hat{D}$, whose restriction to $N_1$  is the APS boundary conditions and whose restriction to $N_2$ is the dual APS boundary conditions. Then $\hat{D}_{\hat{B}}$ is an odd symmetric operator. Our first result, Theorem~\ref{T:splittingtau}, is the following analog of the classical Splitting Theorem for the index
\begin{equation}\label{E:ISplitting}
    \ind_\tau D\ = \ \ind_\tau \hat{D}_{\hat{B}}.
\end{equation}
The proof is similar to the proof for the $\ZZ$-valued index but has the following difficulty. As in the classical case,  the main part of the proof is a construction of a deformation $T_s$ ($s\in[0,1]$) between $D$ and $\hat{D}_{\hat{B}}$ (these are the operators denoted by $\hD^+_{B_s}K_s$ in Section~\ref{SS:KsDs}), cf., for example, \cite[Lemma 8.11, Theorem 8.12]{BarBallmann12}). However, the operators $T_s$ are not odd symmetric in general. In Lemma~\ref{L:kerKsDs}, we show that despite this, the dimension of the kernel of these operators is independent of $s$ modulo 2, which is enough to prove the Splitting Theorem. 

In case when $N$ separates $M$ into two disjoint components $M=M_1\sqcup_N M_2$ our assumptions imply that $\theta^M(M_1)= M_2$. In this case, the index of $\hat{D}_{\hat{B}}$ is equal to the sum of the indices of boundary value problems on $M_1$ and on $M_2$. One checks that the involution $\tau$ sends the kernel of the boundary value problem on $M_2$ to the cokernel of the boundary value problem on $M_1$. This implies, cf. Theorem~\ref{T:indtau=indA1}, that in this case 
\begin{equation}\label{E:Iindtau=indA1}
        \ind_\tau D^+\ \equiv \ \ind (D_1)^+_{APS} \ \equiv \ \ind (D_2)^+_{APS},
\end{equation}
where $\equiv$ stands for equality modulo 2,  $D_j$ is the restriction of $D$ to $M_j$ ($j=1,2$), and  $\ind (D_j)^+_{APS}$ is the usual $\ZZ$-valued index of the Atiyah-Patodi-Singer boundary value problem for $D_j$.  

We can now apply the usual Atiyah-Patodi-Singer index theorem to compute the $\ZZ_2$-valued index of $D$. Suppose all the structures are products near $N\subset M$ and let $A$ denote the restriction of $D$ to $N$. Then $\tau A\tau^{-1}= -A$. It follows that the eta-invariant $\eta(A)$ vanishes and \eqref{E:Iindtau=indA1} and the standard Atiyah-Patodi-Singer index theorem implies 
\begin{equation}
            \ind_\tau D^+\ \equiv \ 
     \int_{M_j}\, \hat{A}(TM)\,\ch(E/S) \ - \frac{\dim\ker A^+}2, \qquad j=1,2.
\end{equation}
By Kramer's degeneracy,  cf. \cite{KleinMartin1952}, the dimension of $\ker A^+$ is even. We conclude that in this situation
\begin{equation}\label{E:Iintegerintegral}
    \int_{M_j}\, \hat{A}(TM)\,\ch(E/S)\ \in \ \ZZ, 
    \qquad j=1,2.
\end{equation}

%------------------------------------------------------------
%------------------------------------------------------------
\section{Boundary value problems for Dirac-type operators}\label{S:index}

In this section, we recall the notion of an elliptic boundary value problem for a Dirac-type operator. The Atiyah-Patody-Singer, the dual Atiyah-Patodi-Singer, and the transmission boundary conditions are elliptic. Our notation and terminology are mostly borrowed from \cite{BarBallmann12}.

%---------------------------------------------
\subsection{A Dirac-type operator}\label{SS:Diractype}
Let $M$ be a compact oriented Riemannian manifold (possibly with boundary) and let $E=E^+\oplus E^-$ be a Dirac bundle over $M$, cf. \cite[Definition~II.5.2]{LawMic89}. In particular, $E$ is a Hermitian vector bundle endowed with a Clifford multiplication $c:T^*M\to \End(E)$ and a compatible Hermitian connection $\n^E$.  

A {\em Dirac-type operator} on $E$ is the differential operator

\begin{equation}\label{E:Dirac}
     D := \ \sum_{1\le j\le n}\, c(e_i)\n^E_{e_j}+V:\,C^\infty(M,E)\to C^\infty(M,E), 
\end{equation}
where $(e_1,\ldots,e_n)$ is an orthonormal frame of $TM$ and $V:E\to E$ is a self-adjoint bundle map. Then $D$ is a formally self-adjoint differential operator, which is independent of the choice of the frame $(e_1,\ldots,e_n)$, cf. \cite[\S3.3]{BeGeVe}.

If $E=E^+\oplus E^-$ is a graded bundle, $E^+$ is orthogonal to $E^-$, the connection preserves the grading, and $c(\xi), V:E^\pm \to E^\mp$ for all $\xi\in T^*M$, we say that $E$ is a graded Dirac bundle. In this case, $D$ takes the form 
\begin{equation}\label{E:AAAgrading}
	D\ = \ \begin{pmatrix}
	0&D^-\\
	D^+&0
	\end{pmatrix},
\end{equation}
where $D^\pm$ is the restriction of $D$ to $C^\infty(M,E^\pm)$.

We denote by $E_\dm= E_\dm^+\oplus E_\dm^-$ the restriction of $E$ to the boundary. 

%---------------------------------------------
\subsection{A boundary-defining function}\label{SS:boundary defining}
If the boundary $\p M$ of $M$ is non-empty, we fix a function $t:M\to [0,\infty)$ such that $t|_{\p M}\equiv 0$ and $dt$ does not vanish on $\p M$. Such a function is called a {\em boundary-defining function}. 

Given a boundary-defining function $t$, we identify a neighborhood $U\subset M$ of the boundary with a cylinder  
\begin{equation}\label{E:Zr}
	Z_r\ := \ [0,r)\times\dm. 
\end{equation}
in such a way that $t$ becomes a projection onto the first factor. Note that, in general, this identification is not an isometry. 

We use the Riemannian metric on $M$ to identify the cotangent bundle to the boundary with a subset of $T^*M$.

%---------------------------------------------
\subsection{An adapted operator on the boundary}\label{SS:adapted operator}
Any Dirac-type operator satisfies the conditions of Lemma~4.1 of \cite{BarBallmann12}. Hence, by this lemma,  there exists a self-adjoint first-order elliptic differential operator $A:C^\infty(\dm,E_\dm)\to C^\infty(\dm,E_\dm)$ such that the restriction of $D$ to the neighborhood $U$, defined in the previous subsection, is given by 
\begin{equation}\label{E:adapted}
    D\ = \ c(dt)\,\left(\, \frac{\p}{\p t}+A+R_t\,\right),
\end{equation}
where $R_t$ is a differential operator on $\dm$ of order at most one, depends smoothly on $t$, and such that $R_0$ is an operator of order zero. In this situation, we refer to $A$ as an {\em adapted operator of $D$} at the boundary. Note that the notion of an adapted operator depends on the choice of the boundary-defining function $t$. One readily sees that the adapted operator is a Dirac-type operator on $E_\dm$. Note, however, that it behaves differently with respect to the grading on $E_\dm$, namely $A:E_\dm^\pm\to E_\dm^\pm$, i.e., with respect to the grading on $E_\dm$ it has the form
\[
	A= \begin{pmatrix}
	    A^+&0\\0&A^-
	\end{pmatrix}:\,C^\infty(\dm,E_{\dm})\;\to\; C^\infty(\dm,E_{\dm})
\]

%----
\begin{remark}\label{E:adapted choice}
The adapted operator $A$ is not unique. Different choices of $A$ lead to various boundary value problems. See the discussion of which choice is convenient for which operator in Sections~7 and 8 of \cite{BrMaschler19}. 
\end{remark}
%------

%---------------------------------------------
\subsection{The product case}\label{SS:product}
We say that the manifold $M$ and the Dirac bundle $E$ are {\em product near the boundary} if the identification between the neighborhood $U\subset M$ of the boundary and the cylinder \eqref{E:Zr} is an isometry and the Clifford multiplication $c:T^*M\to{\rm End}(E)$ and the connection $\nabla^E$ have product structure on $Z_r$. Then the restriction \eqref{E:adapted} of $D$ to $Z_r$ takes the form 
\begin{equation}\label{E:productD}
	D\; = \;c(dt)\left(\frac{\p}{\p t}+A\right).
\end{equation}
Since both $D$ and $A$ are formally self-adjoint operators and $c(dt)$ is skew-adjoint, one sees that $A$ anti-commutes with $c(dt)$:
\begin{equation}\label{E:[A,c]}
	c(dt)\circ A\ = \ -\,A\circ c(dt).
\end{equation}
If \eqref{E:productD} holds, we say that $D$ is {\em a product near the boundary}.

%----------
\subsection{Sobolev spaces on the boundary}\label{SS:Sob&specproj}
The {\em Sobolev space}  $H^s(\dm,E_{\dm})$ is defined to be the completion of $C^\infty(\dm,E_{\dm})$ with respect the norm
\begin{equation}\label{E:Sobnorm}
	\|\bfu\|_{H^s(\dm,E_\pM)}^2\;:=\;
		\big\|(\id+A^2)^{s/2}\bfu\big\|_{L^2(\dm,E_\pM)}^2.
\end{equation}
Though the norm depends on the operator $A$, the Sobolev spaces $H^s(\dm,E_{\dm})$ are independent of $A$, \cite[\S I.7]{ShubinPDObook}.

%---------------------------------------------
\subsection{The hybrid Soblev spaces}\label{SS:hybrid}
The eigensections of $A$ belong to $H^s(\pM,E_\pM)$ for all $s\in \RR$. For $I\subset \RR$ we denote by 
\[
	H_I^s(A)\ \subset \  H^s(\dm,E_{\dm})
\] 
the span of the eigensections of $A$ whose eigenvalues belong to $I$ and by $P_I^A:H^s(\dm,E_{\dm})\to H_I^s(A)$ the orthogonal projection. We set $H_I^s(A^\pm):= H_I^s(A)\cap H^s(\pM,E_\pM^\pm)$.

%-----------------------------------------
\begin{definition}\label{D:hybrid space}
Fix $a\in\RR$ and define the \emph{hybrid} Sobolev space
\begin{equation}\label{E:checkH}
	\cHH(A)\;:=\;
	H_{(-\infty,a]}^{1/2}(A)\;\oplus\;H_{(a,\infty)}^{-1/2}(A)
	\ \subset \ H^{-1/2}(\dm,E_{\dm})
\end{equation}
with $\cHH$-norm
\[
	\|\bfu\|_{\cHH(A)}^2\;:=\;
	\big\|P_{(-\infty,a]}^A\bfu\big\|_{H^{1/2}(\dm,E_\pM)}^2\;+\;
	\big\|P_{(a,\infty)}^A\bfu\big\|_{H^{-1/2}(\dm,E_\pM)}^2.
\]

We set $\cHH(A^\pm)$ the restriction of $\cHH(A)$ to $E^\pm_\dm$.
\end{definition}
The space $\cHH(A)$ is independent of the choice of $a$, \cite[p.~5]{BarBallmann12}. By Theorem~6.7 of \cite{BarBallmann12} the hybrid space
$\cHH(A)$ coincides with the space of restrictions to the boundary of sections of $E$ which lie in the maximal domain of $D$.

Similarly, we define
\begin{equation}\label{E:hatH}
	\hHH(A)\; := \;
	H_{(-\infty,a]}^{-1/2}(A)\;\oplus\;H_{(a,\infty)}^{1/2}(A)
\end{equation}
with $\hHH$-norm
\[
\|\bfu\|_{\hHH(A)}^2\;:=\;
\|P_{(-\infty,a]}^A\bfu\|_{H^{-1/2}(\dm,E_\pM)}^2\;+\;
\|P_{(a,\infty)}^A\bfu\|_{H^{1/2}(\dm,E_\pM)}^2,
\]
and the restrictions $\hHH(A^\pm)$ of this space to $E^\pm_\dm$. 
Then 
\[
	\hHH(A)\;=\;\cHH(-A).
\]

The $L_2$-scalar product on sections of $E$ extends naturally to a perfect pairing
\[
\cHH(A)\;\times\;\hHH(A)\;\to\;\CC.
\]

By \eqref{E:[A,c]}, $c(dt)$ induces an isometry between Sobolev spaces  $H_{(-\infty,a]}^s(A^+)$  and $H_{[-a,\infty)}^s(A^-)$. Therefore, we conclude that

%----------
\begin{lemma}\label{L:admap}
Over $\dm$, the isomorphism $c(dt):E_{\dm}\to E_{\dm}$ induces an isomorphism $\hHH(A^+)\to\cHH(A^-)$. In particular, the sesquilinear form
\[
	\beta:\cHH(A^+)\times\cHH(A^-) \ \to \ \CC,
	\qquad\	\beta(\bfu,\bfv) \ :=\ -\,\big(\bfu,-c(dt)\bfv\big)
	\ =\ -\,\big(c(dt)\bfu,\bfv\big),
\]
is a perfect pairing of topological vector spaces.
\end{lemma}

%-------------------------------------------------------
\subsection{Boundary value problems}\label{SS:boundary value}
It is well known, cf., for example, Theorem~5.22 of \cite{Adams-SobolevSpaces}, that the natural restriction map $\cRR:C^\infty(M,E)\to C^\infty(\dm,E_{\dm})$ extends to a bounded map $\cRR:H^s(M,E)\to H^{s-1/2}(\dm,E_{\dm})$. Let $\dom D_{\max}$ denote the maximal domain of $D$. Theorem~1.7 of \cite{BarBallmann12} states that $\cRR: \dom D_{\max}\to \cHH(A)$. In other words, $\cHH(A)$ can be identified with boundary values of sections in the maximal domain of $D$.

%--------------------
\begin{definition}\label{D:bc}
A closed subspace $B\subset \cHH(A^\pm)$ is called a \emph{boundary condition} for $D^\pm$. We will use the notation $D^+_{B}$ for the operator $D^+$ with the following domains
\[
	\dom(D^\pm_{B})\;=\;\{u\in\dom D^\pm_{\max}:\cRR u\in B\}.
\]
\end{definition}
Thus, $D^\pm_B$ is the operator
\[
    D^\pm_B:\, \dom(D^\pm_{B})\ \to \ L^2(M,E^\mp).
\]
\begin{remark}
In \cite{BarBallmann12}, the operator $D^\pm_B$ is denoted by $D^\pm_{B,\max}$. 
\end{remark}

The space
\begin{equation}\label{E:adbc}
	B^{\rm ad}\;:=\;
	\big\{\,\bfv\in\cHH(A^\mp):\,
		\big(c(dt)\bfu,\bfv\big)=0\mbox{ for all }\bfu\in B\,\big\}
\end{equation}
is called the adjoint boundary condition to $B$. It is a closed subspace of $\cHH(A^\mp)$. By equation (62) of \cite{BarBallmann12}, the operator $D^\mp_{B^{\rm ad}}$ is the adjoint of the operator $D^\pm_B$:
\begin{equation}\label{E:Bad}
    \big(D_B^\pm\big)^*\ = \ D^\mp_{B^{\rm ad}},
\end{equation}
where $D^\mp_{B^{\rm ad}}$ is the operator 
\[
    D^\mp:\, \dom(D^\mp_{B^{\rm ad}})\ \to \ L^2(M,E^\pm).
\]

\begin{definition}\label{D:ellbc}
A boundary condition $B$ is said to be \emph{elliptic} if $B\subset H^{1/2}(A^\pm)$ and $B^{\rm ad}\subset H^{1/2}(A^\mp)$.
\end{definition}

By Corollary~8.6 of \cite{BarBallmann12}, if $B$ is an elliptic boundary condition then  the operators $D^\pm_B$ and $D^\mp_{B^{\rm ad}}$ are Fredholm. We define the index of $D^+_B$ by 
\begin{equation}\label{E:defofind}
    \ind D^\pm_B\ := \ \dim \ker D^\pm_B\ - \ \dim \coker D^\pm_B\ = \ 
    \dim \ker D^\pm_B\ - \ \dim \ker D^\mp_{B^{\rm ad}}.
\end{equation}

%-------------------------------------------------------
\subsection{Atiyah-Patodi-Singer boundary conditions}\label{SS:gAPS}
The boundary conditions
\[
    B(A^\pm) \ := \ H^{1/2}_{(-\infty,0)}(\dm,E_\dm^\pm)
\]
are called the \emph{the Atiyah--Patodi--Singer (APS) boundary conditions} for $D^\pm$. We set
\begin{equation}\label{E:APSbc}
        D^\pm_{APS}\ := \ D^\pm_{B(A^\pm)}.
\end{equation}
We define the {\em dual APS boundary conditions}
\[
    \oB(A^\pm)\ := \ H^{1/2}_{(-\infty,0]}(\dm,E_\dm^\pm). 
\]
One readily checks that $B(A^\pm)$ and $\oB(A^\pm)$
are elliptic boundary conditions (cf. \cite[Example~1.16]{BarBallmann12}) and that 
\begin{equation}\label{E:APS-dAPS}
    \ind D^\pm_{\oB(A^\pm)}\ = \ \ind D^{\pm}_{APS}\ + \ \dim\ker A^\pm,
\end{equation}
(cf. \cite[Corollary~8.8]{BarBallmann12}).

From \eqref{E:adbc} we get
\begin{equation}\label{E:adjointAPS}
    \big(B(A^\pm)\big)^{\ad}\ = \ \oB(A^\mp).
\end{equation}
Therefore, by \eqref{E:defofind}, we now obtain
\begin{equation}\label{E:indexAPS}
  \ind D^\pm_{APS}\ = \ \dim\ker D^\pm_{APS} \ - \ \dim\ker D^\mp_{\oB(A^\mp)}.
\end{equation}

\begin{comment}
Or, using \eqref{E:APS-dAPS},
\begin{equation}\label{E:indexAPS2}
  \ind D^\pm_{APS}\ = \ \dim\ker D^\pm_{APS} \ - \ \dim\ker D^\mp_{B(-A^\mp)} \ - \ \dim\ker A^\mp.
\end{equation}
\end{comment}
%-------------------------------------------------------
\subsection{The Atiyah-Patodi-Singer index theorem}\label{SS:APStheorem}
Suppose that $D$ is a product near the boundary, cf. Section~\ref{SS:product}. Let $R$ denote the Riemannian curvature of the metric $g^M$ and let 
\[
    \hat{A}(TM)\ := \det{}^{1/2}\left( \frac{R/2}{\sinh(R/2)}\right)
\]
be the $\hat{A}$-genus differential form. Let $\ch(E/S)$ denote the relative Chern character form of $E$, cf. \cite[Page~146]{BeGeVe}. Let $\eta(A)$ denote the eta-invariant of the operator $A$. The celebrated Atiyah-Patodi-Singer theorem, \cite{APS1} (see also \cite[Theorem~22.18]{BoosWoj93book}), states that 
\begin{equation}\label{E:APStheorem}
     \ind D^+_{APS}\ = \  \int_M\, \hat{A}(TM)\,\ch(E/S) \ - \ 
     \frac{\eta(A)+\dim\ker A}{2}.
\end{equation}
\begin{remark}
Grubb \cite{Grubb92} suggested extending the Atiyah-Patodi-Singer construction to the case where $D$ is not a product near the boundary. She approached it by deforming all the data near the boundary to those with a product structure.  As a result, the contribution of the boundary to the index consists of two terms -- the $\eta$-invariant and an integral of a certain differential form defined by $D$ over the boundary. Gilkey \cites{Gikey95book,Gilkey93,Gilkey75} used invariance theory to compute this integral for different geometric Dirac operators, including the twisted spin-Dirac and the twisted signature operator. A slightly more unified approach to getting an index theorem for general Dirac-type operators in a non-product case was suggested in \cite{BrMaschler19}. We don't discuss these results in this paper. But an interested reader might easily extend our index theorem \ref{T:tauindex} to the non-product case using these ideas. 
\end{remark}

%-----------
\subsection{The splitting theorem}\label{SS:splittingthm}

Let $M$ be a compact manifold without boundary. Let $N\subset M$ be a hypersurface. Then a neighborhood $U$ of $N$ can be identified with $(-r,r)\times N$. As in Section~\ref{SS:boundary defining} we denote by $t$ the coordinate along the axis of $(-r,r)$. 

Cutting $M$ along $N$ we obtain a manifold $\hM$ (connected or not) with two copies of $N$ as the boundary. So we can write $\hM=(M\setminus N)\sqcup N_1\sqcup N_2$, where a neighborhood of $N_1$ is naturally identified with $[0,r)\times N_1$ and a neighborhood of $N_2$ is naturally identified with $(-r,0]\times N_2$. We set $t_1:=t$ in the neighborhood of $N_1$ and $t_2:= -t$ in the neighborhood of $N_2$. 

The bundle $E=E^+\oplus E^-\to M$ and the operators  $D^\pm$ induce a graded Dirac bundle $\hE=\hE^+\oplus\hE^-\to \hM$ and Dirac-type operators
\[
	\hD^\pm:C^\infty(\hM,\hE^\pm)\ \to \ C^\infty(\hM,\hE^\mp)
\]
on $\hM$.  

As in Section~\ref{SS:adapted operator} we can choose an {\em adapted operator $A:C^\infty(N,E_N)\to C^\infty(N,E_N)$ for the hypersurface $N$}, such that on $U\simeq (-r,r)\times N$
\begin{equation}\label{E:adaptedN}
    D\ = \ c(dt)\,\left(\, \frac{\p}{\p t}+A+R_t\,\right),
\end{equation}
where $R_t$ is a differential operator on $N$ of order at most one, depending smoothly on $t$ and such that $R_0$ is an operator of order zero.
Then $A_1=A$ is an adapted operator for $N_1$ and $A_2=-A$ is an adapted operator for $N_2$ and in neighborhoods of $N_j$ in $\hat{M}$, we have 
\begin{equation}\label{E:productD12}
	D\; = \;c(dt_j)\,\left(\frac{\p}{\p t_j}+A_j+R_{t,j}\,\right),\qquad\text{where}\quad A_1=A, \  A_2=-A, \ R_{t,1}=R_t,\ R_{t,2}=-R_{-t}.
\end{equation}
Thus $\hA=A\oplus(-A)$ is an adapted operator of $\hD$ to $\p \hM$. Hence, 
\begin{equation}\label{E:B(hA)}
       B(\hA)\ = \ B(A)\oplus B(-A). 
\end{equation}

The following {\em Splitting Theorem} is proven in \cite[Theorem 8.17]{BarBallmann12}:

%--------------
\begin{theorem}\label{T:splitting}
Suppose $M, D^+,\hM,( \hD)^+$ are as above. Let $B_1=H_{(-\infty,0)}^{1/2}(A)$ and $B_2= H_{(-\infty,0]}^{1/2}(-A)= H_{[0,\infty)}^{1/2}(A)$ be the APS and the dual APS boundary conditions for $( \hD)^+$ along $N_1$ and $N_2$, respectively. Then $( \hD)^+_{B_1\oplus B_2}$ is a Fredholm operator and
\begin{equation}\label{E:splitting}
	\ind D^+\;=\;\ind( \hD)^+_{B_1\oplus B_2}.
\end{equation}
\end{theorem}

Using \eqref{E:B(hA)} and \eqref{E:APS-dAPS}, we can rewrite \eqref{E:splitting} as 
\begin{equation}\label{E:splitting2}
	\ind D^+\;=\;\ind( \hD)^+_{APS}\ + \ \dim\ker A^+.
\end{equation}

%------------------------------------------------------
%------------------------------------------------------
\section{The $\tau$-index}\label{S:tauindex}

In this section, we recall the construction of the secondary $\ZZ_2$-valued index of odd symmetric operators studied in \cites{Schulz-Baldes15,DeNittisSB15,DollSB21} for bounded operators and in \cite{BrSaeedi24index} for differential operators on a non-compact involutive manifold. 

The exposition in this section follows closely Section~2 of \cite{BrSaeedi24spflow}.

%---------------------------------------------
\subsection{An anti-linear involution}\label{SS:involtiontau}
Let $H$ be a complex Hilbert space and let $\<\cdot,\cdot\>_H$ denote the scalar product on $H$. For an \underline{anti-linear} map $\tau:H\to H$ we denote by $\tau^*$ the unique anti–linear operator satisfying 
\[
	\<\tau x, y\>_H\ = \ \overline{\<x,\tau^*y\>_H} \qquad\text{for all}\quad
	x,y\in H.
\]

%----
\begin{definition}\label{D:antiinvolution}
An anti-linear map  $\tau:H\to H$  is called an {\em anti-involution} if  $\tau^*=-\tau= \tau^{-1}$.
\end{definition}

%----------------------------------
\subsection{Odd symmetric operators}\label{SS:oddsymmetric}
Let $D: H\to H$ be a closed unbounded linear operator on $H$ whose domain is $W$, which is dense in $H$. We view $W$ as a Hilbert space with the scalar product 
\[
	\<x,y\>_W\ = \ \<x,y\>_H\ + \ \<Dx,Dy\>_H.
\]
Then $D:W\to H$ is a bounded operator. 

\begin{definition}\label{D:odd symmetric operator}
A closed operator $D:H\to H$ with a dense domain $W$ is called {\em odd symmetric} (or $\tau$-symmetric) if 
\begin{equation}\label{E:tauDtau}
    \tau\, D\, \tau^{-1}\ = \ D^*,
\end{equation}
where $D^*$ is the adjoint of the unbounded operator $D$. Equation \eqref{E:tauDtau} means, in particular, that $\tau(W)$ is equal to the domain of the operator $D^*$.

We denote by $\calB_\tau(W,H)$ the space of $\tau$-symmetric operators on $H$ whose domain is $W$. Each $T\in \calB_\tau(W,H)$ is a bounded operator from the Hilbert space $W$ to the Hilbert space $H$. The space $\calB_\tau(W,H)$ with the operator norm of these operators is a Banach space.
\end{definition}

%----------------------------------
\subsection{Graded odd symmetric operators}\label{SS:graded oddsymmetric}
Let $H=H^+\oplus H^-$ be a graded Hilbert space and let 
\[
    D \ = \
    \begin{pmatrix}0&D^-\\D^+&0\end{pmatrix}:
    \, H\to H
\]
be a closed operator with dense domain $W= W^+\oplus W^-$. Then $D:H^\pm\to H^\mp$. If  $D$ is self-adjoint, then $(D^\pm)^*= D^\mp$.

Assume that the anti-unitary anti-involution $\tau$ satisfies  $\tau:H^\pm\to H^\mp$. If $D$ is odd symmetric, then \eqref{E:tauDtau} implies that $\tau D^\pm\tau^{-1}= (D^\pm)^*$. In this situation, we say that $D$ is a \textit{graded odd symmetric operator}.

We denote the space of graded odd symmetric self-adjoint operators with domain $W$ by $\widehat{\calB}_\tau(W,H)$.
If $D\in \widehat{\calB}_\tau(W,H)$, then $\tau D^\pm \tau^{-1}= D^\mp= (D^\pm)^*$. Then \textit{the operators $D^\pm$ are odd symmetric}. 

%----------------------------------
\subsection{The $\ZZ_2$-valued index}\label{SS:Z2index}
If $D$ is an odd symmetric operator, then $\ind D=0$ by \eqref{E:tauDtau}. We define the {\em secondary} $\ZZ_2$-valued invariant – the \textit{$\tau$-index of $D$}:

%-------------
\begin{definition}\label{D:tau index}
If $D$ is an odd symmetric operator on $H$ its {\em $\tau$-index} is
\begin{equation}\label{E:tauindex}
	\ind_\tau D\ := \ \dim \ker D \qquad \mod 2.
\end{equation}
If $D\in \widehat{\calB}_\tau(W,H)$ is a graded odd symmetric self-adjoint operator on a graded Hilbert space $H$, we define 
\begin{equation}\label{E:tauindex graded}
	\ind_\tau D^+\ := \ \dim \ker D^+ \qquad \mod 2.
\end{equation}
\end{definition}

Theorem~2.5 of \cite{BrSaeedi24index} states that the $\tau$-index is constant on connected components of the space of odd symmetric Fredholm operators:
%-------------
\begin{theorem}\label{T:homotopyindex}
The $\tau$-index $\ind_\tau D$ is constant on the connected components of $\calB_\tau(W,H)$. The $\tau$-index  $\ind_\tau D^+$ is constant on the connected components of $\widehat{\calB}_\tau(W,H)$.
\end{theorem}

%-----------------
We now discuss some geometric settings leading to many interesting of odd symmetric operators, \cite{BrSaeedi24index}.

%---------------------------------------------
\subsection{Quaternionic vector bundles}\label{SS:quaternionic}
An\textit{ involutive manifold}  $(M,\theta^M)$ is  a Riemannian manifold $M$ together with a metric preserving involution  $\theta^M:M\to M$.  A {\em quaternionic vector bundle} over an involutive manifold $M$ is Hermitian vector bundle $E$ endowed with an anti-linear anti-unitary bundle map $\theta^E:E\to \theta^*E$, such that $(\theta^E)^2=-1$, cf. \cites{Dupont69,DeNittisGomi15,Hayashi17}. This means that for every $x\in M$ there exists an anti-linear map $\theta^E_x:E_x\to E_{\theta(x)}$, which depends smoothly on $x$ and satisfies $\theta^E_{\theta^M(x)}\, \theta^E_x=-1$,
$(\theta^E_x)^*\theta^E_x=1$.

If $E=E^+\oplus E^-$ is a graded vector bundle and  $\theta^E$ is {\em odd with respect to the grading}  (i.e. $\theta^E(E^\pm)=E^\mp$), we say that $(E,\theta^E)$ is a {\em graded quaternionic bundle}.

Given a quaternionic vector bundle $(E,\theta^E)$ over an involutive manifold  $(M,\theta^M)$ we define an anti-unitary anti-involution $\tau$ on the space $\Gamma(M,E)$ of sections of $E$ by 
\begin{equation}\label{E:tau=}
    \tau:\, f(x)\ \mapsto \ \theta^E f(\theta^M x), \qquad f\in\Gamma(M,E).  
\end{equation}
%---------------------------------------------
\subsection{Odd symmetric elliptic operators}\label{SS:oddelliptic}
Let $(E=E^+\oplus E^-,\theta^E)$ be a graded quaternionic vector bundle over an involutive manifold $(M,\theta^M)$. Let $D:\Gamma(M,E)\to \Gamma(M,E)$ be a self-adjoint odd symmetric elliptic differential operator on $M$, which is odd with respect to the grading, i.e. $D:\Gamma(M,E^\pm)\to \Gamma(M,E^\mp)$. We set $D^\pm:= D|_{\Gamma(M,E^\pm)}$. Then $\tau\, D^\pm\, \tau^{-1}= (D^\pm)^*= D^\mp$ and
\[
    D\ = \ \begin{pmatrix}
        0&D^-\\D^+&0
    \end{pmatrix}
\]
with respect to the decomposition $E=E^+\oplus E^-$.

Suppose now that $D$ is Fredholm, then the $\tau$-index $\ind_\tau D^+\in \ZZ_2$ is defined by \eqref{E:tauindex}. In \cite{BrSaeedi24index}, the first and the second authors computed this index in several examples and proved the $\ZZ_2$-valued analogs of the Relative Index Theorem, the Callias index theorem, and Boutet de Monvel's index theorem for Toeplitz operators. 

%------------------------------------------------------
%------------------------------------------------------
\section{The Splitting Theorem for the $\tau$-index}\label{S:splitttingtauindex}

In this section, we prove an analog of the Splitting theorem~\ref{T:splitting} for the $\tau$-index.

%---------------------------------------------
\subsection{An odd symmetric Dirac operator}\label{SS:oddDirac}
Let  $(E,\theta^E)$ be a graded quaternionic bundle over a closed connected oriented involutive manifold $(M,\theta)$, which is also a graded Dirac bundle (no connections between the two structures on $E$ yet). As before, we assume that $\theta^E$ is odd with respect to the grading: $\theta^E:E^\pm\to E^\mp$.

Let $\tau$ be as in \eqref{E:tau=} and let $D$ be a Dirac-type operator on $E$, cf. \eqref{E:Dirac}. Suppose $D$ is odd symmetric,  $\tau\, D\, \tau^{-1}= D$. Then  $D^\pm= \tau\, D^\mp\, \tau^{-1}$.

%-------
\subsection{A $\theta^M$-invariant hypersurface} \label{SS:hypersurface}
Suppose there exists a $\theta^M$-invariant oriented hypersurface $N\subset M$. Then there is a small equivariant neighborhood $U$ of $N$ such that $N$ divides it into two disjoint components $U_1$ and $U_2$. We assume that 
\begin{equation}\label{E:tauU1=U2}
    \theta^M(U_1)= U_2.
\end{equation} 
Then there exists a function $t:U\to \RR$ such that $N=t^{-1}(0)$, $dt$ never vanishes on $N$, and 
\[
    t\big(\theta^M(x)\big) \ = \ -t(x).
\]
A possibly smaller neighborhood of $N$, which by a slight abuse of notations we still denote by $U$, can be identified with the cylinder $(-r,r)\times \dm$ in such a way that 
\[
    \theta^M(t,y)\ = \ \big(-t,\theta^M(y)\big), 
    \qquad t\in (-r,r), \ \ y\in N. 
\]
In $U$ we have $\tau\p_t\tau^{-1}= -\p_t$. Using \eqref{E:adaptedN} and $\tau\, D\, \tau^{-1}= D$ we obtain 
\begin{equation}\label{E:tauct}
    \tau\,c(dt)\,\tau^{-1} \ = \ -\, c(dt).
\end{equation}

%-----
\begin{lemma}\label{L:tausymmetricA}
There exists an adapted operator $A$ for $D$ on $N$ such that
\begin{equation}\label{E:tauA}
    \tau\,A\,\tau^{-1} \ = \ -\,A.
\end{equation}
\end{lemma}
%--
\begin{proof}
Let $\tilde{A}$ be any adapted operator for $D$, so that 
\begin{equation}\label{E:adaptedN'}
    D\ = \ c(dt)\,\left(\, \frac{\p}{\p t}+\tilde{A}+\tilde{R}_t\,\right).
\end{equation}

From \eqref{E:adaptedN} and \eqref{E:tauct} we see that 
\begin{equation}\label{E:tauA+R}
   \tau\,(\tilde{A}+\tilde{R}_t)\,\tau^{-1} 
   \ = \ -\,\tilde{A}-\tilde{R}_{-t}.
\end{equation}
Set 
\[
    A\ := \ \frac{\tilde{A}-\tau \tilde{A}\tau^{-1}}2, \qquad
    R_t\ := \ \frac{\tilde{R}_t-\tau \tilde{R}_{-t}\tau^{-1}}2.
\]
Then $R_t$ is a differential operator of order at most one,  $R_0$ is an operator of order zero, and 
\begin{equation}\label{E:adaptedN''}
    D\ = \ c(dt)\,\left(\, \frac{\p}{\p t}+A+R_t\,\right).
\end{equation}
Hence, $A$ is an adapted operator for $D$ which clearly satisfies \eqref{E:tauA}.
\end{proof}
%----
\begin{remark}\label{E:tauAproduct}
In the product case, when $R_t=0$, any adapted operator satisfies \eqref{E:tauA} by \eqref{E:tauA+R}.
\end{remark}
%---------
\subsection{Cutting along the invariant hypersurface}\label{SS:cuttingN}
Fix an adapted operator $A$ for $D$ which satisfies \eqref{E:tauA}. Let $\hM$, $\hE$, $\hD$,  $N_1$, $N_2$, $t$, $t_1$, $t_2$, $A_1$, and $A_2$ be as in Section~\ref{SS:splittingthm}. Then $\theta^M$ induces a natural involution $\theta^{\hM}$ on $\hM$ such that $\theta^{\hM}:N_1\to N_2$. Similarly, $\theta^E$ induces an anti-unitary anti-involution $\theta^{\hE}$ of $\hE$.

By a slight abuse of notation we denote by $\tau:\Gamma(\hM,\hE)\to \Gamma(\hM,\hE)$ the operator induced by $\theta^{\hM}$ and $\theta^{\hE}$ as in \eqref{E:tau=}. Then 
\begin{equation}\label{E:tauhD}
    \tau \hD \tau^{-1}\ = \ \hD^*,
\end{equation}
where $\hD^*$ denotes the operator formally adjoint to $\hD$. We will say that $\hD$ is {\em formally odd symmetric}. 

From \eqref{E:tauA} and \eqref{E:tauhD} we conclude that
\begin{equation}\label{E:tauA1A2}
    \tau\, A_1\,\tau^{-1}\ = \ A_2.
\end{equation}
Recall that the boundary conditions $B(A)$ and $\oB(A)$ were defined in Section~\ref{SS:gAPS}. It follows from \eqref{E:tauA1A2} that $\tau$ maps $B(A_1)$ to $B(A_2)$ and  $\oB(A_1)$ to $\oB(A_2)$. Recall from Section~\ref{SS:splittingthm} that $\hA= A_1\oplus A_2$ is an adapted operator for $\hD$ and set 
\[
    \hat{B}(\hA)\ := \ B(A_1)\oplus \oB(A_2).
\]

%---
\begin{lemma}\label{L:hDoddsymmetric}
The operator $\hD_{\hat{B}(\hA)}$ is odd symmetric. 
\end{lemma}
\begin{proof}
Equation \eqref{E:adjointAPS} implies that 
\[
    \hat{B}(\hA)^{\ad}\ = \ \oB(A_1)\oplus B(A_2).
\]
Thus $\tau$ sends $\hat{B}(\hA)$ to $\hat{B}(\hA)^\ad$. Hence, the lemma follows from \eqref{E:Bad} and \eqref{E:tauhD}.
\end{proof}

%-----------------------------
\subsection{The Splitting theorem}\label{SS:splittingtau}
We are now ready to formulate the main result of this section:

%---
\begin{theorem}\label{T:splittingtau}
Let $D$ be an odd symmetric Dirac-type operator on a graded quaternionic Dirac bundle $(E=E^+\oplus E^-,\theta^E)$ over a closed involutive manifold $(M,\theta^M)$. Suppose $N\subset M$ is a $\theta^M$-invariant hypersurface such that \eqref{E:tauU1=U2} is satisfied. Let $\hA$, $A_1$, and $A_2$ be as above. Then 
\begin{equation}\label{E:splittingtau}
    \ind_\tau D^+\ = \ \ind_\tau \hD^+_{\hat{B}(\hA)}.
\end{equation}
\end{theorem}

The rest of this section is occupied with the proof of this theorem, which is modeled on the proof of Theorem~5.11 of \cite{BrShi21odd}. 

%-------------------
\subsection{The transmission boundary conditions} \label{SS:transmissionbc}

First, we reformulate the Splitting theorem~\ref{T:splittingtau} solely in terms of the manifold $\hM$. Specifically, we introduce  a boundary condition $B$ on $\hM$ such that $\ind_\tau \hD^+_B= \ind_\tau D^+$ by construction.

Notice that 
\[
    H^{1/2}(\p \hM,E_{\p \hM})\ = \
    H^{1/2}(N_1,E_{N_1})\oplus H^{1/2}(N_2,E_{N_2}). 
\]
We define the {\em transmission boundary conditions} on $\hM$
\begin{equation}\label{E:transmission}
		B_\tr \;:=\;
	\big\{(\bfu,\bfu)\in H^{1/2}(N_1,E_{N_1}^+)\oplus H^{1/2}(N_2,E_{N_2}^+)\big\}. 
\end{equation}
It is shown in Example~1.16(d) of \cite{BarBallmann12} that this is an elliptic boundary condition with $B^\ad_\tr=B_\tr$. Clearly, $\tau:B_\tr\to B_\tr$.

The canonical pull-back of sections from $E$ to $\hE$ identifies the domains of $D^+$ and $\hD^+_{B_\tr}$. It follows that    
\[
    \ind_\tau \hD^+_{B_\tr}\ = \ \ind_\tau D^+.
\]
Hence, the Splitting theorem is equivalent to the equality
\begin{equation}\label{E:splittingtauB}
    \ind_\tau \hD^+_{B_\tr}\ = \ \ind_\tau ( \hD)^+_{\hat{B}(\hA)}.
\end{equation}
Below we prove this equality by constructing a homotopy between the operators $\hD^+_{B_\tr}$ and $( \hD)^+_{\hat{B}(\hA)}$.

%-------------
\subsection{A deformation of boundary conditions}
\label{SS:deformationbc}
We construct a continuous family of elliptic boundary conditions connecting $\hat{B}(\hA)$ and the transmission boundary condition $B$. 

Any $\bfu\in H^{1/2}(N,E_N^+)$ can be decomposed as $\bfu=\bfu^++ \bfu^-$, where $\bfu^+\in H^{1/2}_{[0,\infty)}(N,E_N^+)$ and $\bfu^-\in H^{1/2}_{(-\infty,0)}(N,E_N^+)$. Then $\bfu^+\in \oB(A_2)$ and $\bfu^-\in B(A_1)$. Define two families of boundary conditions on $\hat{N}:=N_1\sqcup N_2$:
\[
	B_{s}\ := \ 
	\Big\{\, 
    \big(\bfu^-+(1-s)\bfu^+,\bfu^++(1-s)\bfu^-\big):\,
  \bfu\in H^{1/2}(N,E_{N}^+)\,\Big\}.
\]
Then $B_s\subset H^{1/2}(N_1,E_{N_1}^+)\oplus H^{1/2}(N_2,E_{N_2}^+)$ and the  adjoint of $B_s$ is the boundary condition 
\[
        B_s^{\rm ad}\\
        =\; \Big\{(\bfv^-+(1-s)\bfv^+,\bfv^++
        (1-s)\bfv^-)\Big\}
        \subset H^{1/2}(\p\hM,E_{\p\hM}^-),
\]
where $\bfv^-\in H_{(-\infty,0]}^{1/2}(N,E_{N}^-)$ and $\bfv^+\in H_{(0,\infty)}^{1/2}(N,E_{N}^-)$. Thus $B_s$ is an elliptic boundary condition for all $s\in[0,1]$ and we get a family of Fredholm operators $\{\hD_{B_s}\}_{0\le s\le1}$.

By \eqref{E:tauA1A2}, 
\begin{equation}\label{E:tauB+B}
    \tau:\, B_s\ \to \  B_s^{\rm ad}.
\end{equation}
Hence, 
\begin{equation}\label{E:tauhDB+B}
    \tau\, \hD^+_{B_s}\, \tau^{-1}\ = \ 
    \hD^-_{B_s^{\rm ad}}.
\end{equation}
By \eqref{E:Bad}, the right-hand side of this equality is equal to the adjoint of $\hD^+_{B_s^{\rm ad}}$. Hence, the operators $\hD^+_{B_s}$ are odd symmetric, for all $s\in [0,1]$.

%--------------
\subsection{A family of operators with the same boundary conditions}
\label{SS:family of operators}
Clearly, $\hat{B}(\hA) = B_1$ and the transmission boundary conditions $B= B_0$. Thus the family $\hD^+_{B_s}$ interpolates between $\hD^+_{\hat{B}(\hA)}$ and $\hD^+_{B}$. However, the operators $D^+_{B_s}$ have different domains, so Theorem~\ref{T:homotopyindex}, can no be applied directly. We finish the proof of Theorem~\ref{T:splittingtau} by a slight modification of the method of \cite[Lemma 8.11, Theorem 8.12]{BarBallmann12}: we construct a continuous family of operators 
\[
    K_s:\,L^2(\hM,\hE^+)\to L^2(\hM,\hE^+)
\]
such that the restriction 
\[
    K_s\;:\;\dom \hD^+_{B_\tr}\;\to\;\dom \hD^+_{B_s}
\]
is an isomorphism. Then   
\begin{equation}\label{E:DKs}
    \hD^+_{B_s}\circ K_s:\,  \dom \hD^+_{B_\tr}
    \ \to \  L^2\big(\hM,\hE^-\big)
\end{equation}
is a continuous family of operators. These operators are not odd symmetric. But in Lemma~\ref{L:kerKsDs} we show that $\dim\ker \hD^+_{B_s}\circ K_s$ is independent of $s$ modulo 2. Hence, 
\begin{multline}\label{E:indDK=indD}
    \ind_\tau \hD^+_{B_\tr}\ \equiv \ \ind_\tau \hD^+_{B_1} 
    \ \equiv \ \dim\ker \hD^+_{B_1}\circ K_1
    \\ \equiv \ \dim\ker \hD^+_{B_0}\circ K_0
    \ \equiv \ \ind_\tau \hD^+_{B_0}
    \ \equiv \ \ind_\tau \hD^+_{\hat{B}(\hA)},
\end{multline}
where, as usual, ``$\equiv$" denotes equality modulo 2.
This proves \eqref{E:splittingtauB} and, hence, Theorem~\ref{T:homotopyindex}. 

We finish the proof by constructing the family $K_s$ and showing that the dimension of the kernel of $\hD^+_{B_s}\circ K_s$ is independent of $s$ modulo 2.

%--------------
\subsection{Family of isomorphisms of boundary conditions}
\label{SS:familykt}
Recall that we denote by $B_\tr$ the transmission boundary conditions \eqref{E:transmission} and that $B_0=B_\tr$. First, consider a family of isomorphisms
\[
	k_s:B_\tr\to B_s,\qquad 
	k_s(\bfu,\bfu):=(\bfu^-+(1-s)\bfu^+,\bfu^++(1-s)\bfu^-).
\]
Then $k_0=\id$ and $k_s$ are isomorphisms from $B_\tr$ to $B_s$. 

By definition,
\[
	(k_{s_1}-k_{s_2})(\bfu,\bfu)\;=\;(s_2-s_1)(\bfu^+,\bfu^-).
\]
Notice that $\|(\bfu^+,\bfu^-)\|_{H^{1/2}(\p\hM,\hE_{\p\hM})}\le\|(\bfu,\bfu)\|_{H_{}^{1/2}(\p\hM,\hE_{\p\hM})}$. 
Hence, for $s_1,s_2\in[0,1]$ with $|s_1-s_2|<\varepsilon$,  the operator
\[
	k_{s_1}-k_{s_2}:\, B_\tr \to \ H^{1/2}(\p\hM,\hE_{\p\hM})
\]
has a norm not greater than $\varepsilon$. This implies that $\{k_s\}$ is a continuous family of maps from $B_\tr$ to $H_{}^{1/2}(\p\hM,\hE_{\p\hM})$. 

%--------------
\subsection{The extension map}\label{SS:extension map}
Fix a $\theta^M$-invariant smooth cut-off function $\chi:M\to [0,1]$ such that $\chi(t)=1$ in a small neighborhood of $N$ and $\chi(t)=0$ outside of $U\simeq(-r,r)\times N$. By slight abuse of notation, we denote by $\chi$ also the induced function on $\hM$.

Let  $B\in H^{1/2}(\p\hM,\hE^+_{\p\hM})$ be an elliptic boundary condition. 
By Lemma~7.3 of \cite{BarBallmann12} the domain $\dom \hD_B$ coincides with the space which is denoted by $H^1_D(M,E^+;B)$ in \cite{BarBallmann12}. Hence, it follows from Theorem~6.7(iii) of \cite{BarBallmann12} that the formula
\[
    (\E\phi)(t,y)\ := \ \chi(t)\cdot\exp\big(-|t||A|\big)\,\phi,
    \qquad  t\in (-r,r), \ y \in N,
\]
defines a continuous linear map,
\[
    \E: B\ \to \ \dom\hD^+_B,
\]
where
$\dom\hD^+_B$ is viewed as a Hilbert space with the scalar product 
\begin{equation}\label{E:scalarproductdom}
    \<\phi,\psi\>_{\dom\hD^+_B}\ := \
    \<\phi,\psi\>_{L^2(M,E^+)}\ + \ \<\hD_B\phi,\hD_B\psi\>_{L^2(M,E^-)}.
\end{equation}
The map $\E$ is called the {\em extension map}.

%--------------
\subsection{A family of operators with the same domain}
\label{SS:KsDs}
Let $B_s$ be the family of boundary conditions defined in Section~\ref{SS:deformationbc} and $k_s$ be as in Section~\ref{SS:familykt}. Define a family of maps
\begin{equation}\label{E:Ks}
    K_s:\, \dom(\hD^+_{B_0})\ \to \ \dom(\hD^+_{B_s}),\qquad
    K_s(\phi)\ := \ \phi\ + \ \E\big(k_s(\cRR(\phi)) - \cRR(\phi)\big)
\end{equation}
It is shown in the proof of  Theorem~8.12 of \cite{BarBallmann12} that $K_s$ is a continuous family of isomorphisms of Hilbert spaces and that the compositions
\[
    \dom(\hD^+_{B_0}) \ \overset{K_s}{\longrightarrow} \  \dom(\hD^+_{B_s})
   \  \overset{\hD^+_{B_s}}{\longrightarrow} \ L^2(M,E^-)
\]
is a continuous family of bounded operators from the Hilbert space $\dom(\hD^+_{B_0})$ to $L^2(M,E^-)$. 

The composition $\hD^+_{B_s}K_s$ is not odd symmetric. However, the following lemma is true

%----
\begin{lemma}\label{L:kerKsDs}
The dimension of the kernel of 
\begin{equation}\label{E:kerhDKs}
    \hD^+_{B_s}K_s:\, \dom(\hD^+_{B_0})\ \to \ L^2(M,E^-)
\end{equation}
is independent of $s$ modulo 2.
\end{lemma}
\begin{proof}

Let $K_s^\dagger:\dom(\hD^+_{B_s})\to \dom(\hD^+_{B_0})$ denote the adjoint of $K_s$. Here, for an unbounded operator $A$ on $L^2(M,E)$ we reserve the notation $A^*$ for its adjoint with respect to the $L^2$-scalar products, and denote by $A^\dagger$ the adjoint with respect to some other scalar products, which will be specified in each case. In particular, $K_s^\dagger$ is defined with respect to the graph scalar products on $\dom(\hD^+_{B_s})$ and $\dom(\hD^+_{B_0})$.

Set 
\[
    U_s\ := \ K_s(K_s^\dagger K_s)^{-1/2} : \, \dom(D_{B_0})\,\ \to \ \dom(D_{B_s})
\]
is a family of unitary operators, $U_s^\dagger = U_s^{-1}$, and the composition 
\begin{equation}\label{E:UsDs}
    \hD^+_{B_s}U_s :\, \, \dom(\hD^+_{B_0})\ \to \ L^2(M,E^-)
\end{equation}
is a continuous family of bounded maps from the Hilbert space $\dom(\hD^+_{B_0})$ to $L^2(M,E^-)$. It is enough to show that the dimension of the kernel of these maps is constant modulo 2. 

Recall that the adjoint $(\hD^+_{B_s})^*$ of the unbounded operator $\hD^+_{B_s}$ is equal to $\hD^-_{B_s^\ad}$. Let 
\[
    (\hD_{B_s}^+)^\dagger:\,L^2(M,E^-)\ \to \ \dom(\hD^+_{B_s})
\]
denote the adjoint of the bounded operator 
\[
    \hD^+_{B_s}:\,\dom(\hD^+_{B_s}) \to \ L^2(M,E^-)
\]
between the Hilbert spaces $\dom(\hD^+_{B_s})$ and $L^2(M,E^-)$.
Then for $\phi\in \dom (\hD^+_{B_s})$ and $\psi\in \dom(\hD^-_{B_s^\ad})$ we have
\begin{multline}\notag
    \Big\<\, \phi, \hD^-_{B_s^\ad}\psi\,\Big\>_{L^2(M,E^+)} \ = \ 
    \Big\<\, \hD^+_{B_s}\phi, \psi\,\Big\>_{L^2(M,E^-)} \ = \
   \Big\<\,\, \phi, (\hD^+_{B_s})^\dagger\psi\,\Big\>_{\dom(\hD^+_{B_s})} \\ = \ 
    \Big\<\,\, \phi, (\hD^+_{B_s})^\dagger\psi\,\Big\>_{L^2(M,E^+)}\ + \
    \Big\<\,\, \hD^+_{B_s}\phi, \hD^+_{B_s}(\hD^+_{B_s})^\dagger\psi\,\Big\>_{L^2(M,E^-)}
    \\ = \ 
    \Big\<\, \phi, (1+\hD^-_{B_s^\ad}\hD^+_{B_s})(\hD^+_{B_s})^\dagger\psi\,\Big\>_{L^2(M,E^+)}.
\end{multline}
Since $\dom(\hD^-_{B_s^\ad})$ is dense in $L^2(M,E^-)$, comparing the left and right-hand sides of this equality gives
\begin{equation}\label{E:Ds*=Ds+}
    (\hD_{B_s}^+)^\dagger
    \ = \ \Big(1+\hD^-_{B_s^\ad}\hD^+_{B_s}\Big)^{-1}\hD^-_{B_s^\ad}.
\end{equation}

Let $\big(\hD^+_{B_s} U_s\big)^\dagger:L^2(M,E^-)\to \dom(\hD^+_{B_0})$ denote the adjoint of the operator \eqref{E:UsDs}. Using $U_s^\dagger=U_s^{-1}$ and \eqref{E:Ds*=Ds+}, we obtain
\[
    \big(\hD^+_{B_s} U_s\big)^\dagger\ = \ 
    U_s^{-1}\Big(1+\hD^-_{B_s^\ad}\hD^+_{B_s}\Big)^{-1}\hD^-_{B_s^\ad}.
\]

We now compute the dimension of the kernel of $\hD^+_{B_s}K_s$ as
\begin{multline}\label{E:dimkerlong}
    \dim\ker \hD^+_{B_s}K_s\ = \ \dim\ker \hD^+_{B_s}U_s\ = \ 
    \dim\ker \big(\hD^+_{B_s}U_s\big)^\dagger\hD^+_{B_s}U_s\\ = \ 
    \dim\ker U_s^{-1}\Big(1+\hD^-_{B_s^\ad}\hD^+_{B_s}\Big)^{-1}\hD^-_{B_s^\ad} \hD^+_{B_s}U_s.
\end{multline}
Here, the last two operators are positive self-adjoint endomorphisms of the Hilbert space $\dom(D^+_{B_0})$. 

We are now ready to prove the lemma. Set
\[ 
    P_s\ := \ \big(\hD^+_{B_s}U_s\big)^\dagger\hD^+_{B_s}U_s\ = \ 
   U_s^{-1}\Big(1+\hD^-_{B_s^\ad}\hD^+_{B_s}\Big)^{-1}\hD^-_{B_s^\ad} \hD^+_{B_s}U_s.
\]
It follows from \eqref{E:dimkerlong} that it suffices to show that for every $s_0\in[0,1]$ there exists $\epsilon>0$ such that $\dim\ker P_s$ is independent of $s$ modulo 2 for $s\in (s_0-\epsilon,s_0+\epsilon)$.

Since $D_s^*D_s$ is an elliptic operator of positive order and the boundary conditions $B_s$ and $B_s^\ad$ are elliptic, the spectrum of the operator $\hD^-_{B_s^\ad} \hD^+_{B_s}$ is discrete. Therefore,  the intersection of the spectrum of the operator $\Big(1+\hD^-_{B_s^\ad}\hD^+_{B_s}\Big)^{-1}\hD^-_{B_s^\ad} \hD^+_{B_s}$ and, hence, of $P_s$, 
with the interval $[0,1)$ is discrete. Hence, 
for $s_0\in[0,1]$, there exists $\epsilon, \delta>0$ such that $\delta<1$ is not in the spectrum of $P_s$ for all $s\in (s_0-\epsilon,s_0+\epsilon)$ and the intersection of the spectrum of $P_{s_0}$ with the interval $[0,\delta]$ is equal to $\{0\}$. Then 
for all $s\in (s_0-\epsilon,s_0+\epsilon)$
\[
    \dim \ker P_{s_0} \ = \
    \sum_{0\le\lambda<\delta} \dim E_{s,\lambda}   ,
\]
where $ E_{s,\lambda}$ is the eigenspace of $P_s$ with eigenvalue $
\lambda$. 

Hence, to prove the lemma, it is enough to show that the dimension of $E_{s,\lambda}$ is even for all $0<\lambda<\delta$.

Set $\tilde{E}_{s,\lambda}:= U_s E_{s,\lambda}$. Then, for all $\phi\in \tilde{E}_{s,\lambda}$,
\[
    \Big(1+\hD^-_{B_s^\ad}\hD^+_{B_s}\Big)^{-1}\hD^-_{B_s^\ad} \hD^+_{B_s}\,\phi\ = \ \lambda\phi.
\]
It follows that 
\[
    \hD^-_{B_s^\ad} \hD^+_{B_s}\,\phi\ = \ \frac{\lambda}{1-\lambda}\,\phi.
\]
Hence, $\tilde{E}_{s,\lambda}$ is an eigenspace of $\hD^-_{B_s^\ad} \hD^+_{B_s}$. To finish the proof of the lemma we need to show that 
\[
    \dim \tilde{E}_{s,\lambda}\ = \ \dim E_{s,\lambda}
\]
is even. We proceed as in the proof of Theorem~2.5 of \cite{BrSaeedi24index}. Set $\mu= \frac{\lambda}{1-\lambda}$.

Recall that $\tau \hD^+_{B_s}\tau^{-1}=\hD^-_{B_s^\ad}$ and $\tau^{-1}=-\tau$. Hence, 
\[
    \hD^-_{B_s^\ad} \hD^+_{B_s}\,\tau \hD^+_{B_s}\ = \ \tau \hD^+_{B_s}\,  \hD^-_{B_s^\ad}\hD^+_{B_s}.
\]
It follows that $\tau \hD^+_{B_s}(\tilde{E}_{s,\lambda})\subset \tilde{E}_{s,\lambda}$.  For $\lambda>0$  consider the operator 
\[
    \calE_{s,\lambda}:= \frac1{\sqrt{\mu}}\tau \hD^+_{B_s}:\,\tilde{E}_{s,\lambda}\ \to \tilde{E}_{s,\lambda}.
\]
Then 
\[
    \calE_{s,\lambda}^2 \ = \  \frac1\mu\, \tau \hD^+_{B_s} \tau \hD^+_{B_s}
    \ = \ - \frac1\mu\tau \hD^+_{B_s} \tau^{-1} \hD^+_{B_s}
    \ = \ - \frac1\mu \hD^-_{B_s^\ad}\hD^+_{B_s} \ = \ -\ID:\,\tilde{E}_{s,\lambda}\ \to \tilde{E}_{s,\lambda}.
\]
Thus $\calE_{s,\lambda}:\, \tilde{E}_{s,\lambda}\ \to \tilde{E}_{s,\lambda}$ is an anti-unitary anti-involution. From Kramers' degeneracy (cf., for example,  Lemma~2.2 of \cite{BrSaeedi24index}) we conclude now that the dimension of $\tilde{E}_{s,\lambda}$ is even for all $\lambda\in(0,\delta)$.
\end{proof}

%--------------
\subsection{Proof of Theorem~\ref{T:splittingtau}}
\label{SS:prsplittingtau}
Theorem~\ref{T:splittingtau} follows now from \eqref{E:indDK=indD}, as explained in the end of Section~\ref{SS:family of operators}.
\hfill$\square$

%------------------------------------------------
%-----------------------------------------------
\section{The case of a separating hypersurface and a $\tau$-index theorem}\label{S:tauindextheorem}

As in Section~\ref{S:splitttingtauindex} we assume that $(E,\theta^E)$ is a quaternionic Dirac bundle over a closed involutive manifold $(M,\theta^M)$ and that the corresponding Dirac-type operator $D$ is odd symmetric.

%------
\subsection{A separating hypersurface}\label{S:deparating}
We assume there exists a $\theta^M$-invarint hypersurface $N\subset M$ which divides $M$ in two, i.e, that 
\[
    M\ = \ M_1\sqcup_N M_2,
\]
where $M_1$ and $M_2$ are compact manifolds whose boundary is identified with $N$ such that $M_1\cap M_2= N$. We also assume that $\theta^M(M_1)= M_2$. Note that if the involution $\theta^M$ does not have fixed points such a hypersurface always exists and is called a {\em characteristic submanifold}, cf. \cite[\S I.2.1]{Medrano71boookinvolutions}.

As in Section~\ref{S:splitttingtauindex}, we identify a neighborhood of $N$ with the cylinder $(-r,r)\times N$ so that $\theta^M(t,y)= \big(-t,\theta^M(y)\big)$  ($t\in (-r,r), \ y\in N$). Then, by \eqref{E:tauct}, $\tau$ anti-commutes with $c(dt)$.

%----------------
\subsection{The splitting theorem in case of a separating hypersurface}\label{SS:splittingseparating}
The manifold $\hM$  introduced in Section~\ref{SS:splittingthm} now becomes a disjoint union
\[
    \hM\ = \ M_1 \sqcup M_2.
\]

Let $A$ be an adapted operator for $D$ on $N$ which satisfies condition \eqref{E:tauA} of Lemma~\ref{L:tausymmetricA}.  Let $D_1$ and $D_2$ denote the restrictions of $D$ to $M_1$ and $M_2$ respectively. Note that they are also equal to the restrictions of the operators $\hD$ to these manifolds. Using the notation of Section~\ref{SS:cuttingN}, we have 
\[
    \hD^+_{\hat{B}(\hA)}\ = \ 
    (D_1^+)_{B(A_1)}\oplus  (D_2^+)_{\oB(A_2)}.
\]
Thus the Splitting Theorem~\ref{T:splittingtau}, becomes
\begin{equation}\label{E:splittingM1M2}
    \ind_\tau D^+\ = \ 
    \ind_\tau (D_1^+)_{B(A_1)}\ + \ 
    \ind_\tau (D_2^+)_{\oB(A_2)}.
\end{equation}

Next, we show that this equality implies that $\ind_\tau D^+$ is equal to the mod 2 reduction of the usual index of the APS boundary value problem for operator $D_j$ ($j=1,2$). Recall, that $(D_j)^+_{APS}$
is by definition equal to $(D_j)^+_{B(A_j)}$, cf. \eqref{E:APSbc}.
%-----------
\begin{theorem}\label{T:indtau=indA1}
In the situation described above
\begin{equation}\label{E:indtau=indA1}
        \ind_\tau D^+\ \equiv \ \ind (D_1)^+_{APS} \ \equiv \ \ind (D_2)^+_{APS}
\end{equation}
where $\equiv$ denotes equality modulo 2.
\end{theorem}

%----------
\begin{proof}
By \eqref{E:Bad}, the adjoint operator to $(D_2)^+_{\oB(A_2)}$ is $(D_1)^-_{B(A_1)}$. Hence, 
\begin{equation}\label{E:indtau=coker}
    \ind_\tau (D_2)^+_{\oB(A_2)}\ \equiv \ \dim\ker (D_2)^+_{\oB(A_2)}\ = \ \dim\coker (D_1)^-_{B(A_1)}.
\end{equation}
Since $\ind_\tau (D_1)^+_{B(A_1)} = \dim\ker(D_1)^+_{B(A_1)}$ we conclude from \eqref{E:splittingM1M2} and \eqref{E:indtau=coker} that
\begin{multline}\label{E:indD=indD1}
    \ind_\tau D^+\ \equiv \ 
   \dim\ker(D_1)^+_{B(A_1)}\ + \ 
    \dim\coker (D_1)^-_{B(A_1)} 
    \\ \equiv \ 
    \dim\ker(D_1)^+_{B(A_1)}\ - \ \dim\coker (D_1)^-_{B(A_1)}
    \ = \ \ind (D_1)^+_{B(A_1)}.
\end{multline}
This proves the first equality in \eqref{E:indtau=indA1}.

The second equality in \eqref{E:indtau=indA1} follows from the fact that the choice of what submanifold was labeled $M_1$ and what was labeled $M_2$ was completely random. Hence, \eqref{E:indD=indD1} also implies that $\ind_\tau D^+\ \equiv \ \ind (D_2)^+_{B(A_2)}$.
\end{proof}

%----------------------------------
\subsection{A $\tau$-index theorem in the product case}\label{SS:tauindextheorem}
We now combine Theorem~\ref{T:indtau=indA1} and the APS index formula \eqref{E:APStheorem} to express the $\tau$-index of $D^+$ as an integral of a differential form in the case when all the structures are products near the boundary. In particular, we assume that near $N$ the Dirac operator $D$ is given by \eqref{E:productD}.

From \eqref{E:tauA}, we conclude that 
\[
        \eta(A)\ = \ \eta(\tau A\tau^{-1})\ = \ -\,\eta(A).
\]
Hence, 
\begin{equation}\label{E:eta=0}
    \eta(A)\ = \ 0.
\end{equation}

Combining Theorem~\ref{T:indtau=indA1}, the APS index formula \eqref{E:APStheorem}, and \eqref{E:eta=0} we obtain the following 

%------------
\begin{theorem}\label{T:tauindex}
Let $D$ be an odd symmetric Dirac-type operator on a graded quaternionic Dirac bundle $(E=E^+\oplus E^-,\theta^E)$ over a closed involutive manifold $(M,\theta^M)$. Suppose $N\subset M$ is a $\theta^M$-invariant hypersurface, which divides $M$ into two components: $M= M_1\sqcup_N M_2$ and all the structures are product near $N$. Then 
\begin{multline}\label{E:tauindextheorem}
        \ind_\tau D^+\ \equiv \ 
     \int_{M_1}\, \hat{A}(TM)\,\ch(E/S) \ - \frac{\dim\ker A^+}2
     \\ \equiv \ \int_{M_2}\, \hat{A}(TM)\,\ch(E/S)
      \ - \frac{\dim\ker A^+}2,
\end{multline}
where $\equiv$ denotes the equality modulo 2. 
\end{theorem}

%----
By \eqref{E:tauct}, the product $\sigma= \tau{}c(dt)$ is an anti-linear anti-involution. From \eqref{E:[A,c]} and \eqref{E:tauA} it follows that this anti-involution commutes with $A$ and preserves the grading: $\sigma A^+\sigma^{-1}= A^+$. Hence, $\sigma$ acts on $\ker A$. By Kramer's degeneracy,  cf. \cite{KleinMartin1952}  or \cite[Lemma~2.2]{BrSaeedi24index}, the dimension of $\ker A^+$ is even. Thus 
we obtain the following corollary of Theorem~\ref{T:indtau=indA1}

%----
\begin{corollary}\label{C:integerintegral}
Under the conditions of Theorem~\ref{T:indtau=indA1},
\begin{equation}\label{E:integerintegral}
    \int_{M_j}\, \hat{A}(TM)\,\ch(E/S)\ \in \ \ZZ, 
    \qquad j=1,2.
\end{equation}

\end{corollary}

%------------------------------------------------
%-----------------------------------------------
% \bib, bibdiv, biblist are defined by the amsrefs package.

\end{document}